\documentclass{article}

\usepackage{amsfonts}
\usepackage{amssymb}
\usepackage{amsthm}
\usepackage{amsmath}

\newtheorem{thm}{Theorem}[section]
\newtheorem{lem}[thm]{Lemma}
\newtheorem{prop}[thm]{Proposition}
\newtheorem{cor}[thm]{Corollary}

\newtheorem{defn}{Definition}[section]

\newtheorem{ques}{Question}[section]
\newtheorem{prob}{Problem}[section]

\usepackage{amsfonts}
\usepackage{amssymb}
\usepackage{amsthm}
\usepackage{amsmath}

\newcommand{\A}{\mathcal{A}}
\newcommand{\B}{\mathcal{B}}

\begin{document}

\title{Interpreting a field in its Heisenberg group}

\author{R.\ Alvir, W.\ Calvert, G.\ Goodman, V.\ Harizanov, J.\ Knight, \\
A.\ Morozov, R.\ Miller, A.\ Soskova, and R.\ Weisshaar\footnote{The first, second, fourth, fifth, sixth, seventh and ninth authors are grateful for support from NSF grant DMS \#1600625.  The first, third, and fifth authors also acknowledge support from NSF grant DMS \#1800692.  The fourth author was partially supported by Grant \# 429466 from the Simons Foundation. The seventh author was partially supported by Grant \# 581896 from the Simons Foundation and by the City University of New York PSC-CUNY Research Award Program. The eighth author was partially supported by BNSF, DN 02/16,  and SU Science Fund, 80-10-128/16.04.2020.
}} 

\date{\today}

\maketitle

\begin{abstract}

We improve on and generalize a 1960 result of Maltsev.  For a field $F$, we denote by $H(F)$ the Heisenberg group with entries in $F$.  Maltsev 
showed that there is a copy of $F$ defined in $H(F)$, using existential formulas with an arbitrary non-commuting pair $(u,v)$ as parameters.  We show that $F$ is interpreted in $H(F)$ using computable $\Sigma_1$ formulas with no parameters.  We give two proofs.  The first is an existence proof, relying on a result of Harrison-Trainor, Melnikov, R.\ Miller, and Montalb\'{a}n.  
This proof allows the possibility that the elements of $F$ are represented by tuples in $H(F)$ of no fixed arity.  The second proof is direct, giving explicit finitary existential formulas that define the interpretation, with elements of $F$ represented by triples in 
$H(F)$.  Looking at what was used to arrive at this parameter-free interpretation of $F$ in $H(F)$, we give general conditions sufficient to eliminate parameters from interpretations.

\end{abstract}

\section{Introduction}
\label{sec:intro}

The Heisenberg group of a field $F$ is the upper-triangular subgroup of $GL_3(F)$ in which all
matrices have $1$'s along the diagonal and $0$'s below it.
Maltsev showed that there are existential formulas with parameters, which, for every field $F$, define $F$ in its Heisenberg group $H(F)$.  In this article we will produce existential formulas without parameters, which, for every field $F$, interpret $F$ in $H(F)$. Observing what is used to obtain this result, we will then formulate a general result on removing parameters from an interpretation.

Languages are assumed to be computable, and structures are assumed to have universe a subset of $\omega$.  For a given structure $\mathcal{A}$, the atomic diagram $D(\mathcal{A})$ may be identified, via G\"{o}del numbering, with a subset of $\omega$.  We then identify $\mathcal{A}$ itself with the characteristic function of $D(\mathcal{A})$.  Classes of structures have a fixed language, and are closed under isomorphism.  The following notion, of ``Turing computable embedding,'' is from \cite{CCKM}, based on the earlier notion
of ``Borel embedding'' from \cite{FS}.  

\begin{defn}

For classes $K,K'$, we say that $K$ is \emph{Turing computably embedded} in $K'$, and we write $K\leq_{tc} K'$, if there is a Turing operator $\Theta:K\rightarrow K'$ such that for all $\mathcal{A},\mathcal{B}\in K$, $\mathcal{A}\cong\mathcal{B}$ iff $\Theta(\mathcal{A})\cong\Theta(\mathcal{B})$.  

\end{defn}

Medvedev reducibility is used to compare ``problems,'' where a \emph{problem} is a subset of $\omega^\omega$.  The problems that concern us have the form ``build a copy of $\mathcal{A}$.''

\begin{defn}
  
For structures $\mathcal{A}$ and $\mathcal{B}$, we say that $\mathcal{A}$ is \emph{Medvedev reducible} to $\mathcal{B}$, and we write $\mathcal{A}\leq_s\mathcal{B}$, 
if there is a Turing operator $\Phi$ that takes copies of $\mathcal{B}$ to copies of $\mathcal{A}$.

\end{defn}

We are interested in ``uniform'' Medvedev reductions, which, for a given Turing computable embedding $\Theta$, take any copy of a structure in the range of $\Theta$ to a copy of its pre-image.     

\begin{defn}
  
Let $\Theta$ be a Turing computable embedding of a class $K$ to a class $K'$.  We say that the structures in $K$ are \emph{uniformly Medvedev reducible} to their $\Theta$-images in $K'$, if there is a Turing operator $\Phi$ such that for all $\mathcal{A}\in K$, $\Phi$ serves as a Medvedev reduction of $\mathcal{A}$ to $\Theta(\mathcal{A})$.  

\end{defn}
    
Often, when we have a Turing computable embedding $\Theta:K\rightarrow K'$ with a uniform Medvedev reduction of the structures in $K$ to their $\Theta$-images, it is because there are simple formulas that define, for all $\mathcal{A}\in K$, an interpretation of $\mathcal{A}$ in $\Theta(\mathcal{A})$.  Montalb\'{a}n defined a very general kind of interpretation of $\mathcal{A}$ in $\mathcal{B}$ that yields a uniform Medvedev reduction of $\mathcal{A}$ to $\mathcal{B}$.  In this definition, the tuples from $\mathcal{B}$ that represent elements of $\mathcal{A}$ may have arbitrary arity.  The interpretation is defined by formulas that have no specific arity. Here, the arity of a formula is the number of its free variables.  As usual, we often write $\mathcal{B}$ both for the structure and its domain.

\begin{defn} [Generalized computable $\Sigma_1$-definition]

Let $R\subseteq\mathcal{B}^{<\omega}$, and let $\varphi_n(\bar{x}_n)_{n\in\omega}$ be a computable sequence of computable $\Sigma_1$ formulas, where $\varphi_n(\bar{x}_n)$ has arity $n$.  If for each $n$, $\varphi_n(\bar{x}_n)$ defines $R\cap\mathcal{B}^n$, then we say that $\bigvee_n\varphi_n(\bar{x}_n)$ is a \emph{generalized computable $\Sigma_1$ definition} of $R$.    

\end{defn}   

Since a generalized computable $\Sigma_1$ formula allows consideration of tuples of all finite arities, it is technically not in $L_{\omega_1\omega}$; however, it is a computable disjunction, over all $n\in\omega$, of $L_{\omega_1\omega}$ formulas $\varphi_n$ with free variables $x_1,\ldots,x_n$.  Generalized computable $\Sigma_1$ formulas are involved in the following definition.

\begin{defn} [Montalb\'{a}n]

For a relational structure $\mathcal{A} = (A,(R_i)_{i \in I})$ and a structure $\mathcal{B}$, we say 
$\mathcal{A}$ is \emph{effectively interpreted} in $\mathcal{B}$ if there exist a set 
$D\subseteq\mathcal{B}^{<\omega}$ and relations $\sim$ and $R_i^*$ on $D$ such that  

\begin{enumerate}

\item  $(D,(R_i^*)_{i \in I})/_\sim\cong\mathcal{A}$, 

\item  there is a computable sequence of generalized computable $\Sigma_1$ formulas, with no parameters, defining the set $D$ and the following relations on $D$:  $\sim$ and the complementary relation $\not\sim$, and for each $i$, the relation $R_i^*$ and the complementary relation $\neg{R_i^*}$.

\end{enumerate}

\end{defn} 

\noindent
\textbf{Notation and terminology}:  We may later simply write $\pm \sim$ (or $\pm R_i^*$) for the complementary pair of relations $\sim$ and $\not\sim$ (or $R_i^*$ and $\neg{R_i^*}$).  We may think of the pair of generalized computable $\Sigma_1$ formulas that define the complementary pair pair $\pm\sim$ (or $\pm R_i^*$) as a generalized $\Delta_1$ definition of $\sim$ (or $R_i^*$).  

\bigskip
\noindent
\textbf{Remark}:  In the Russian tradition, a structure that is effectively interpreted in $\B$ is said to be \emph{$\Sigma$-definable in $\B$}.

\bigskip

Below, we illustrate the use of tuples of arbitrary arity.

\begin{prop}
\label{computable}

If $\mathcal{A}$ is computable, then it is effectively interpreted in all structures $\mathcal{B}$.

\end{prop}

\begin{proof}

Let $D = \mathcal{B}^{<\omega}$.  Let $\bar{b}\sim\bar{c}$ if $\bar{b},\bar{c}$ are tuples of the same length.  For simplicity, suppose $\mathcal{A} = (\omega,R)$, where $R$ is binary.  If $\mathcal{A}\models R(m,n)$, let $R^*(\bar{b},\bar{c})$ for all $\bar{b}$ of length $m$ and $\bar{c}$ of length $n$.  Then $(D,R^*)/_\sim\cong \mathcal{A}$. 
\end{proof}  
 
The following definition was first presented as \cite[Defn.\ 3.1]{MPSS}.   

\begin{defn} 

A \emph{computable functor} from $\mathcal{B}$ to $\mathcal{A}$ is a pair of Turing operators 
$\Phi,\Psi$ such that:  

\begin{enumerate}

\item  $\Phi$ takes copies of $\mathcal{B}$ to copies of $\mathcal{A}$,

\item  $\Psi$ takes each triple $(\mathcal{B}_1,f,\mathcal{B}_2)$ such that $\mathcal{B}_i\cong\mathcal{B}$ for $i=1,2$ and $\mathcal{B}_1\cong_f\mathcal{B}_2$ to a function $g$ such that $\Phi(\mathcal{B}_1)\cong_g\Phi(\mathcal{B}_2)$.  Moreover, $\Psi$ preserves identity and composition.  

\end{enumerate}

\end{defn}

Harrison-Trainor, Melnikov, Miller, and Montalb\'{a}n \cite{HTM^3} proved the following.  

\begin{thm}
\label{HTM^3}

For a pair of structures $\mathcal{A}$ and $\mathcal{B}$, the following are equivalent:

\begin{enumerate}

\item  $\mathcal{A}$ is effectively interpreted in $\mathcal{B}$,

\item   there is a computable functor from $\mathcal{B}$ to $\mathcal{A}$. 

\end{enumerate}

\end{thm}

\noindent
\textbf{Remarks}:  In the proof of Theorem \ref{HTM^3}, it is important that $D$ consist of tuples of arbitrary arity.  Proposition \ref{computable} said that a computable structure $\mathcal{A}$ can be effectively interpreted in an arbitrary structure $\mathcal{B}$.  We proved this by a direct construction, in which $D$ was the set of all tuples from $\mathcal{B}$.  There is an alternative proof of Proposition \ref{computable}, using Theorem \ref{HTM^3}.  We define a computable functor $\Phi,\Psi$ from $\mathcal{B}$ to $\mathcal{A}$ in which $\Phi$ ignores the oracle and simply computes $\A$, while $\Psi$ always computes the identity function.    

\bigskip

We are interested in uniform effective interpretations and uniform computable functors.   

\begin{defn}

Suppose $K\leq_{tc} K'$ via $\Theta$.  The structures in $K$ are \emph{uniformly effectively interpreted} in their $\Theta$-images if there is a fixed collection of generalized computable $\Sigma_1$ formulas (without parameters) that, for all $\mathcal{A}\in K$, define an interpretation of $\mathcal{A}$ in $\Theta(\mathcal{A})$.  

\end{defn}

\begin{defn}
  
Suppose $K\leq_{tc} K'$ via $\Theta$.  Turing operators $\Phi$ and $\Psi$ form a \emph{uniform computable functor} from the structures in the range of $\Theta$ to their pre-images provided that for all $\mathcal{A}\in K$, $\Phi$ and $\Psi$ serve as a computable functor from $\Theta(\mathcal{A})$ to $\mathcal{A}$.  

\end{defn}

There is a uniform version of Theorem \ref{HTM^3}.  

\begin{thm}
\label{uniform}

For classes $K,K'$ with $K\leq_{tc} K'$ via $\Theta$, the following are equivalent:

\begin{enumerate}

\item  there is a uniform effective interpretation of the structures $\mathcal{A}\in K$ in the corresponding structures $\Theta(\mathcal{A})$, 

\item  there is a uniform computable functor $\Phi,\Psi$ from the structures $\Theta(\mathcal{A})$ in the range of $\Theta$ to their pre-images $\mathcal{A}$.  

\end{enumerate}

\end{thm}     

It is natural to ask whether, when $\mathcal{A}\leq_s\mathcal{B}$, there must be an effective interpretation of 
$\mathcal{A}$ in $\mathcal{B}$.  It is also natural to ask whether, when $\mathcal{A}$ is effectively interpreted in 
$(\mathcal{B},\bar{b})$ with parameters $\bar{b}$, it must be effectively interpreted in $\mathcal{B}$ without parameters.  Kalimullin \cite{Kalimullin} gave examples providing negative answers to both questions. 

\bigskip

Maltsev defined a Turing computable embedding of fields in $2$-step nilpotent groups.  The embedding takes each field $F$ to its \emph{Heisenberg group} $H(F)$.  To show that the embedding preserves isomorphism, Maltsev gave uniform existential formulas defining a copy of $F$ in $H(F)$.  The definitions involved a pair of parameters, whose orbit is defined by an existential formula (in fact, the formula is quantifier-free).  In Section \ref{sec:Maltsev}, we recall Maltsev's definitions.  In Section \ref{sec:Morozov}, we describe a uniform computable functor that, for all $F$, takes copies of $H(F)$, with their isomorphisms, to copies of $F$, with corresponding isomorphisms.  By Theorem~\ref{uniform}, it follows that there is a uniform effective interpretation of $F$ in $H(F)$ with no parameters.  In Section \ref{sec:finitary}, we give explicit finitary existential formulas that define such an interpretation, and also show that parameter-free interpretations necessarily involve an equivalence relation $\sim$ distinct from equality.  (Thus, while one can interpret $F$ in $H(F)$ without parameters, one cannot define $F$ in $H(F)$ without parameters.)  In Section \ref{sec:biinterpretability}, we note that although $F$ is effectively interpretable in $H(F)$ and $H(F)$ is effectively interpretable in $F$, we do not, in general, have effective bi-interpretability.  In Section \ref{sec:generalize}, we generalize our process of passing from Maltsev's definition, with parameters, to the uniform effective interpretation, with no parameters.

\section{Defining $F$ in $H(F)$}
\label{sec:Maltsev}

In this section, we recall Maltsev's embedding of fields in $2$-step nilpotent groups, and his formulas that define a copy of the field in the group.  Recall that for a field $F$, the Heisenberg group $H(F)$ is the set of matrices of the form
\[h(a,b,c) = \left[\begin{array}{ccc}
1 & a & c\\
0 & 1 & b\\
0 & 0 & 1
\end{array}\right]\]
with entries in $F$.  Note that $h(0,0,0)$ is the identity matrix.  We are interested in non-commuting pairs in $H(F)$.  One such pair is $(h(1,0,0),h(0,1,0))$.  For $u = h(u_1,u_2,u_3)$ and $v = h(v_1,v_2,v_3)$, let 
\[\Delta_{(u,v)} = \left|\begin{array}{cc}
u_1 & v_1\\
u_2 & v_2
\end{array}
\right|.\ \]

For a group $G$, we write $Z(G)$ for the center.  
For group elements $x,y$, the \emph{commutator} is $[x,y] = x^{-1}y^{-1}xy$.  The following technical lemma provides much of the information we need to show that $F$ is defined, with parameters, in $H(F)$.                

\begin{lem}\
\label{lemma:firstmain}

\begin{enumerate}

\item  \begin{enumerate}

\item  For $u$ and $v$, the commutator, $[u,v]$, is $h(0,0,\Delta_{(u,v)})$, and 

\item  $[u,v] = 1$ iff $\Delta_{(u,v)} = 0$. 

\end{enumerate}

\item  Let $u = h(u_1,u_2,u_3)$, and let $v = h(v_1,v_2,v_3)$.  If $\left[\begin{array}{c} u_1\\ u_2\end{array}\right] = \left[\begin{array}{c}
0\\ 0\end{array}\right]$, then $u\in Z(H(F))$.  If $\left[\begin{array}{c} u_1\\ u_2\end{array}\right]\not=\left[\begin{array}{c} 0\\ 0\end{array}\right]$, then $[u,v] = 1$ iff there exists $\alpha$ such that $\left[\begin{array}{c} v_1\\ v_2\end{array}\right] = 
\alpha\cdot\left[\begin{array}{c} u_1\\ u_2\end{array}\right]$.

\item  The center $Z(H(F))$ consists of the elements of the form $h(0,0,c)$.

\item  If $[u,v]\not= 1$, then $x\in Z(H(F))$ iff $[x,u] = [x,v] = 1$.

\end{enumerate} 

\end{lem} 

\begin{proof}

For Part 1, (a) is proved by direct computation, and (b) follows from (a).  Parts 2 and 3 are easy consequences of Part 1.  We prove Part 4.  Suppose $[u,v]\not= 1$.  If $x\in Z(H(F))$, then it commutes with both $u$ and $v$.  We must show that if $x$ commutes with both $u$ and $v$, then $x\in Z(H(F))$.  Let $u = h(u_1,u_2,u_3)$, $v = h(v_1,v_2,v_3)$, and $x = h(x_1,x_2,x_3)$.  By Part 2, since $[x,u] = 1$, there exists $\alpha$ such that $\left[
\begin{array}{c}
x_1\\
x_2\end{array}\right] = \alpha\left[\begin{array}{c}
u_1\\
u_2
\end{array}
\right]$.  Similarly, since $[x,v] = 1$, there exists $\beta$ such that 
$\left[
\begin{array}{c}
x_1\\
x_2\end{array}\right] = \beta\left[\begin{array}{c}
v_1\\
v_2
\end{array}
\right]$.  Since the vectors $\left[
\begin{array}{c}
u_1\\
u_2\end{array}\right]$ and $\left[\begin{array}{c}
v_1\\
v_2
\end{array}
\right]$, are linearly independent, this implies that $\alpha = \beta = 0$.  It follows that $x_1 = x_2 = 0$, so 
$x\in Z(H)$.        
\end{proof}

\begin{cor}
\label{cor:orbits}
If $x\in H(F)$ is fixed by all automorphisms of $H(F)$, then $x=1$.
%
\end{cor}
\begin{proof}
Write $x=h(a,b,c)$.  Lemma \ref{lemma:firstmain}(3) shows $a=b=0$,
since all conjugations fix $x$.  But the automorphism
of $H(F)$ mapping $h(x,y,z)$ to $h(y,x, xy-z)$, which interchanges $h(1,0,0)$
with $h(0,1,0)$, maps $h(0,0,c)$ to $h(0,0,-c)$, hence shows that $c=0$ as well.
\end{proof}

The next lemma tells us how, for any non-commuting pair $u,v$ in the group
$(H(F),*)$, we can define operations $+$ and 
$\cdot$, and an isomorphism $f$ from $F$ to $(Z(H(F)),+,\cdot)$.      

\begin{lem}
\label{MainLemma}

Let $u = h(u_1,u_2,u_3)$ and $v = h(v_1,v_2,v_3)$ be a non-commuting pair. Assume that  $\alpha,\beta,\gamma\in F$. Let 
$x = h(0,0,\alpha\cdot\Delta_{(u,v)})$, $y = h(0,0,\beta\cdot\Delta_{(u,v)})$, and 
$z = h(0,0,\gamma\cdot\Delta_{(u,v)})$.  Then 

\begin{enumerate}

\item  $\alpha + \beta = \gamma$ iff $x*y = z$, where $*$ is the matrix multiplication.

\item  $\alpha\cdot\beta = \gamma$ iff there exist $x'$ and $y'$ such that $[x',u] = [y',v] = 1$, 
$[u,y'] = y$, $[x',v] = x$, and $z = [x',y']$.

\end{enumerate}

\end{lem}

\begin{proof}

For Part 1, matrix multiplication yields the fact that \[h(0,0,a)*h(0,0,b) = h(0,0,a+b)\ .\]  
Then $\alpha + \beta = \gamma$ iff 
\[x*y = h(0,0,\alpha\cdot \Delta_{(u,v)})*h(0,0,\beta\cdot\Delta_{(u,v)}) = h(0,0,\gamma\cdot\Delta_{(u,v)}) = z\ .\]  
For Part 2, first suppose that $\alpha\cdot\beta = \gamma$.  We take $x' = h(\alpha\cdot u_1, \alpha\cdot u_2,0)$, and 
$y' = h(\beta\cdot v_1, \beta\cdot v_2,0)$.  Then $\Delta_{(x',u)} = 0$, so $[x',u] = h(0,0,0) = 1$.  
Similarly, $[y',v] = 1$.  Also, $\Delta_{(x',v)} = \alpha\cdot \Delta_{(u,v)}$, so $[x',v] = h(0,0,\alpha\cdot\Delta_{(u,v)}) = x$.  
Similarly, $\Delta_{(u,y')} = \beta\cdot\Delta_{(u,v)}$, so $[u,y'] = h(0,0,\beta\cdot\Delta_{(u,v)}) = y$.  Finally, 
$\Delta_{(x',y')} = \alpha\cdot\beta\cdot\Delta_{(u,v)} = \gamma\cdot\Delta_{(u,v)}$, so $[x',y'] = h(0,0,\gamma\cdot\Delta_{(u,v)}) = z$.
  
Now, suppose we have $x'$ and $y'$ such that $[x',u] = [y',v] = 1$, $[u,y'] = y$, 
$[x',v] = x$, and $[x',y'] = z$.  Say that $x' = h(x'_1,x'_2,x'_3)$ and $y' = h(y'_1,y'_2,y'_3)$.  Since $[x',v] = x$, $\Delta_{(x',v)} = \alpha\cdot\Delta_{(u,v)}$, so $\left[\begin{array}{c}
x'_1\\  x'_2\end{array}\right] = \alpha\left[\begin{array}{c} u_1\\  u_2\end{array}\right]$.  
Since $[u,y'] = y$, $\Delta_{(u,y')} = \beta\cdot\Delta_{(u,v)}$, so $\left[\begin{array}{c}
y'_1\\ y'_2\end{array}\right] = 
\beta\left[\begin{array}{c} v_1\\ v_2\end{array}\right]$.  
Combining these facts, we see that $\Delta_{(x',y')} = \left|\begin{array}{cc} x'_1 & y'_1\\  x'_2 & y'_2\end{array}\right| = \left|\begin{array}{cc} \alpha\cdot u_1 & \beta\cdot v_1\\  \alpha\cdot u_2 & \beta\cdot v_2\end{array}\right| = \alpha\cdot\beta\cdot\Delta_{(u,v)}$.  Since $[x',y'] = z$, $\Delta_{(x',y')} = \gamma\cdot\Delta_{(u,v)}$.  Since $u$ and $v$ do not commute, $\Delta_{(u,v)}\not= 0$.  Therefore, $\alpha\cdot\beta = \gamma$.  
\end{proof}

The main result of the section follows directly from Lemmas \ref{lemma:firstmain} and \ref{MainLemma}.

\begin{thm} [Maltsev, Morozov]
\label{Mal'tsevDef}

For an arbitrary non-commuting pair $(u,v)$ in $H(F)$, we get $F_{(u,v)} = (Z(H(F)),\oplus,\otimes_{(u,v)})$ where 

\begin{enumerate} 

\item $x\in Z(H(F))$ iff $[x,u] = [x,v] = 1$, 

\item  $\oplus$ is the group operation from $H(F)$,

\item  $\otimes_{(u,v)}$ is the set of triples $(x,y,z)$ such that there exist $x',y'$ with \\
$[x',u] = [y',v] = 1$, 
$[x',v] = x$, $[u,y'] = y$, and $[x',y'] = z$,

\item  the function $g_{(u,v)}$ taking $\alpha\in F$ to $h(0,0,\alpha\cdot\Delta_{(u,v)})\in H(F)$ is an isomorphism between $F$ and $F_{(u,v)}$. 

\end{enumerate}

\end{thm} 

\noindent
\textbf{Note}:  From Part 4, it is clear that $h(0,0,\Delta_{(u,v)})$ is the multiplicative identity in $F_{(u,v)}$---we may write $1_{(u,v)}$ for this element.   

\begin{prop}  

There is a uniform Medvedev reduction $\Phi$ of $F$ to $H(F)$.    

\end{prop}

\begin{proof}

Given $G\cong H(F)$, we search for a non-commuting pair $(u,v)$ in $G$, and then use Maltsev's definitions to get a copy of $F$ computable from $G$.  
\end{proof} 

It turns out that the Medvedev reduction $\Phi$ is half of a computable functor.  In the next section, we explain how to get the other half.     

\section{The computable functor}
\label{sec:Morozov}

In the previous section, we saw that for any field $F$ and any non-commuting pair $(u,v)$ in $H(F)$, there is an isomorphic copy $F_{(u,v)}$ of $F$ defined in $H(F)$ by finitary existential formulas with parameters $(u,v)$.  The defining formulas are the same for all $F$.  Hence, there is a uniform Turing operator $\Phi$ that, for all fields $F$, takes copies of $H(F)$ to copies of $F$.  In this section, we describe a companion operator $\Psi$ so that $\Phi$ and $\Psi$ together form a uniform computable functor.  For any field $F$, and any triple $(G_1,p,G_2)$ such that $G_1$ and $G_2$ are copies of $H(F)$ and $p$ is an isomorphism from $G_1$ onto $G_2$, the function $\Psi(G_1,p,G_2)$ must be an isomorphism from $\Phi(G_1)$ onto $\Phi(G_2)$, and, moreover, the isomorphisms given by $\Psi$ must preserve identity and composition.  We saw in the previous section that for any field $F$, and any non-commuting pair $(u,v)$ in $H(F)$, the function $g_{(u,v)}$ taking $\alpha$ to $h(0,0,\alpha\cdot\Delta_{(u,v)})$ is an isomorphism from $F$ onto $F_{(u,v)}$.  We use this $g_{(u,v)}$ below.       

\begin{lem}
\label{Functorial}

For any $F$ and any non-commuting pairs $(u,v)$, $(u',v')$ in $H(F)$, there is a natural isomorphism $f_{(u,v),(u',v')}$ from $F_{(u,v)}$ onto $F_{(u',v')}$.  Moreover, the family of isomorphisms $f_{(u,v),(u',v')}$ is functorial; i.e., 

\begin{enumerate}

\item for any non-commuting pair $(u,v)$, the function $f_{(u,v),(u,v)}$ is the identity,

\item  for any three non-commuting pairs $(u,v)$, $(u',v')$, and $(u'',v'')$, 
\[f_{(u,v),(u'',v'')} = f_{(u',v'),(u'',v'')}\circ f_{(u,v),(u',v')}.\]

\end{enumerate}

\end{lem}

\begin{proof}

We let $f_{(u,v),(u',v')} = g_{(u',v')}\circ g_{(u,v)}^{-1}$.  This is an isomorphism from $F_{(u,v)}$ onto $F_{(u',v')}$.  It is clear that $f_{(u,v),(u,v)}$ is the identity.  Consider non-commuting pairs $(u,v)$, $(u',v')$, and $(u'',v'')$.  We must show that $f_{(u',v'),(u'',v'')}\circ f_{(u,v),(u',v')} = f_{(u,v),(u'',v'')}$.
We have:
\begin{eqnarray*}
f_{(u',v'),(u'',v'')}\circ f_{(u,v),(u',v')}&=&
g_{(u'',v'')}\circ g_{(u',v')}^{-1}\circ g_{(u',v')}\circ g_{(u,v)}^{-1}= \\
&=&g_{(u'',v'')}\circ g_{(u,v)}^{-1}=\\
&=&f_{(u,v),(u'',v'')}.
\end{eqnarray*}
\end{proof} 

The next lemma says that there is a uniform existential definition of the family of isomorphisms $f_{(u,v),(u',v')}$.   

\begin{lem} 
\label{Existential}

There is a finitary existential formula $\psi(u,v,u',v',x,y)$ that, for any two non-commuting pairs $(u,v)$ and $(u',v')$, defines the isomorphism $f_{(u,v),(u',v')}$ taking $x\in F_{(u,v)}$ to $y\in F_{(u',v')}$.    

\end{lem} 

\begin{proof}

Since the operation $\otimes_{(u,v)}$ and $1_{(u',v')}$ are definable by $\exists$--formulas with parameters $u,v$ and $u',v'$ respectively, it suffices to prove the equivalence
$$f_{(u,v),(u',v')}(x)=y\Leftrightarrow x\otimes_{(u,v)}1_{(u',v')}=y.$$
First assume that
$f_{(u,v),(u',v')}(x)=y$, i.e.,
$y=g_{(u',v')}\circ g_{(u,v)}^{-1}(x)$.
Let $\alpha=g_{(u,v)}^{-1}(x)$, i.e., $x=h(0,0,\alpha\cdot \Delta_{(u,v)})$.
It follows that
$y=h\left(0,0,\alpha\cdot \Delta_{(u',v')}\right)$.
Then
\begin{eqnarray*}
x\otimes_{(u,v)}1_{(u'v')}&=&
 h\left(0,0,\alpha\cdot \Delta_{(u,v)}\right)\otimes_{(u,v)}h\left(0,0,\Delta_{(u',v')}\right)= \\
 &=&h\left(0,0,\alpha\cdot \Delta_{(u,v)}\right)\otimes_{(u,v)}h\left(0,0,\frac{\Delta_{(u',v')}}{\Delta_{(u,v)}}\cdot \Delta_{(u,v)}\right)= \\
\\
 &=&h\left(0,0,\alpha\cdot \frac{\Delta_{(u',v')}}{\Delta_{(u,v)}}\cdot \Delta_{(u,v)}\right)= \\
 &=&h\left(0,0,\alpha\cdot \Delta_{(u',v')}\right)=y.
\end{eqnarray*}
Assume now that
$x\otimes_{(u,v)}1_{(u',v')}=y$ and let $x=h\left(0,0,\alpha\cdot \Delta_{(u,v)}\right)$.
Then
\begin{eqnarray*}
y&=&x\otimes_{(u,v)}1_{(u',v')}=h\left(0,0,\alpha\cdot \Delta_{(u,v)}\right)\otimes_{(u,v)}h\left(0,0,\Delta_{(u',v')}\right)= \\
&=&h\left(0,0,\alpha\cdot \Delta_{(u',v')}\right)=g_{(u',v')}\circ g_{(u,v)}^{-1}(x)=f_{(u,v),(u',v')}(x).
\end{eqnarray*}       
\end{proof} 

We will use Lemmas \ref{Functorial} and \ref{Existential} to prove the following.       

\begin{prop}

There is a uniform computable functor that, for all fields $F$, takes $H(F)$ to $F$.   

\end{prop}

\begin{proof}

Let $\Phi$ be the uniform Medvedev reduction of $F$ to $H(F)$.\  Take copies $G_1,G_2$ of $H(F)$ and take $p$ such that $G_1\cong_p G_2$.\  We describe $q = \Psi(G_1,p,G_2)$ as follows.  Let $(u,v)$ be the first non-commuting pair in $G_1$, and let $(u',v')$ be the first non-commuting pair in $G_2$.\  Now, $p$ takes $(u,v)$ to a non-commuting pair $(p(u),p(v))$, and $p$ maps $F_{(u,v)}$ isomorphically onto $F_{(p(u),p(v))}$.\  The function $f_{(p(u),p(v)),(u',v')}$ is an isomorphism from $F_{(p(u),p(v))}$ onto $F_{(u',v')}$.\  We get an isomorphism $q$ from $F_{(u,v)}$ onto $F_{(u',v')}$ by composing $p$ with $f_{(p(u),p(v)),(u',v')}$.\  For $x\in F_{(u,v)}$, we let $q(x) = f_{(p(u),p(v)),(u',v')}(p(x))$.\  Since $f_{(p(u),p(v)),(u',v')}$ is defined by an existential formula, with parameters $p(u),p(v),u',v'$, we can apply a uniform effective procedure to compute $q$ from $(G_1,p,G_2)$.  

If $G_1 = G_2$ and $p$ is the identity, then $(u,v) = (u',v')$, and by Lemma \ref{Functorial}, $f_{(u,v),(u',v')}$ is the identity. Consider $G_1,G_2,G_3$, all copies of $G$, with functions $p_1,p_2$ such that $G_1\cong_{p_1} G_2$ and $G_2\cong_{p_2} G_3$.\  Then $p_3 = p_2\circ p_1$ is an isomorphism from $G_1$ onto $G_3$. Let $q_1 = \Psi(G_1,p_1,G_2)$, $q_2 = \Psi(G_2,p_2,G_3)$, and $q_3 = \Psi(G_1,p_3,G_3)$.\   We must show that $q_3 = q_2\circ q_1$.\  The idea is to transfer everything to $G_3$ and use Lemma \ref{Functorial}. Let $r_1$ be the result of transferring $q_1$ down to $G_3$---$r_1 = f_{(p_3(u),p_3(v)),(p_2(u’),p_2(v’))}$.\   We have $q_1(x) = y$ iff $r_1(p_3(x)) = p_2(y)$.\  Let $r_2$ be the result of transferring $q_2$ down to $G_3$---$r_2 = f_{(p_2(u’),p_2(v’)),(u'',v'')}$.\  We have $q_2(y) = z$ iff $r_2(p_2(y)) = z$.\  We let $r_3$ be the result of transferring $q_3$ down to $G_3$---$r_3 = f_{(p_3(u),p_3(v)),(u'',v'')}$.\  We have $q_3(x) = z$ iff $r_3(p_3(x)) = z$.\  By Lemma \ref{Functorial}, $r_3 = r_2\circ r_1$.\  If $q_1(x) = y$ and $q_2(y) = z$, then $r_1(p_3(x)) = p_2(y)$, and $r_2(p_2(y)) = z$.\  Then $r_3(p_3(x)) = z$, so $q_3(x) = z$, as required.      
\end{proof}

\begin{cor}
\label{cor:infinterpretation}
There is a uniform effective interpretation of $F$ in $H(F)$.
\end{cor}

\begin{proof}
Apply the result from \cite{HTM^3}.  
\end{proof}

The result from \cite{HTM^3} gives a uniform interpretation of $F$ in $H(F)$, valid for all countable fields
$F$, using computable $\Sigma_1$ formulas with no parameters. The tuples from $H(F)$ that represent
elements of $F$ may have arbitrary arity.  In the next section, we will do better.

We note here that the uniform interpretation of $F$ in $H(F)$ given in this section
allows one to transfer the computable-model-theoretic properties of any graph $G$
to a 2-step-nilpotent group, without introducing any constants.  This is not a new
result:  in \cite{Mekler}, Mekler gave a related coding of graphs into 2-step-nilpotent
groups, which, in concert with the completeness of graphs for such properties
(see \cite{HKSS}), appears to yield the same fact, although Mekler's coding
had different goals than completeness.  Then, in \cite{HKSS}, Hirschfeldt,
Khoussainov, Shore, and Slinko used Maltsev's interpretation of an integral domain
in its Heisenberg group with two parameters, along with the completeness of
integral domains, to re-establish it.  More recently, \cite{MPSS} demonstrated
the completeness of fields, by coding graphs into fields,  From that result,
along with Corollary \ref{cor:infinterpretation} and the usual definition of $H(F)$
as a matrix group given by a set of triples from $F$, we achieve a coding
of graphs into 2-step-nilpotent groups, different from Mekler's coding, with no constants required.


\section{Defining the interpretation directly}
\label{sec:finitary}

Our goal in this section is to give explicit existential formulas defining a uniform effective
interpretation of a field in its Heisenberg group.  We discovered the formulas for this interpretation
by examining the infinitary formulas used in the interpretation in Corollary \ref{cor:infinterpretation}
and trimming them down to their essence, which turned out to be finitary.

\begin{thm}
\label{thm:parameterfree}

There are finitary existential formulas that, uniformly for every field $F$, define an effective interpretation
of $F$ in $H(F)$, with elements of $F$ represented by triples of elements from $H(F)$.

\end{thm}


We offer intuition before giving the formal proof.  The domain $D$ of the interpretation
will consist of those triples $(u,v,x)$ from $H(F)$ with $uv\neq vu$ and $x$ in the center:
for each single $(u,v)$, we apply Maltsev's definitions, with $u$, $v$ as parameters, to get  
$F_{(u,v)}\cong F$.
We view the triples arranged as follows:\\
\setlength{\unitlength}{0.2in}
\begin{picture}(18,9)(-3,0)

\put(-1,7){\line(1,0){19}}
\multiput(-1,7)(5,0){4}{\line(0,-1){6}}

\put(1,7.5){$F_{(u,v)}$}
\put(6,7.5){$F_{(u',v')}$}
\put(11,7.5){$F_{(u'',v'')}$}
\put(15.5,7.5){$\cdots$}

\put(0,6){$(u,v,x_0)$}
\put(0,5){$(u,v,x_1)$}
\put(0,4){$(u,v,x_2)$}
\put(0,3){$(u,v,x_3)$}
\put(1.5,2){$\vdots$}

\put(5,6){$(u',v',x_0)$}
\put(5,5){$(u',v',x_1)$}
\put(5,4){$(u',v',x_2)$}
\put(5,3){$(u',v',x_3)$}
\put(6.5,2){$\vdots$}

\put(10,6){$(u'',v'',x_0)$}
\put(10,5){$(u'',v'',x_1)$}
\put(10,4){$(u'',v'',x_2)$}
\put(10,3){$(u'',v'',x_3)$}
\put(11.5,2){$\vdots$}

\end{picture}

Here each column can be seen as $F_{(u,v)}$ for some non-commuting pair $(u,v)$.
Now the system of isomorphisms from Lemma \ref{Functorial} will allow us to identify each element
in one column with a single element from each other column, and modding out by this identification
will yield a single copy of $F$.

\begin{proof}

Let $H$ be a group isomorphic to $H(F)$.  Recalling the natural isomorphisms
$f_{(u,v),(u',v')}$ defined in Lemma \ref{Functorial} for non-commuting pairs $(u,v)$ and $(u',v')$,
we define $D\subseteq H$, a binary relation $\sim$ on $D$, and ternary relations $\oplus$, $\odot$
(which are binary operations) on $D$, as follows.

\begin{enumerate}

\item  $D$ is the set of triples $(u,v,x)$ such that $uv\neq vu$ and $xu=ux$ and $xv=vx$.
(Notice that, no matter which non-commuting pair $(u,v)$ is chosen, the set of corresponding
elements $x$ is precisely the center $Z(H)$, by Theorem~\ref{Mal'tsevDef}.) 

\item  $(u,v,x)\sim (u',v',x')$ holds if and only if the isomorphism 
$f_{(u,v),(u',v')}$ from $F_{(u,v)}$ to $F_{(u',v')}$ maps $x$ to $x'$. 

\item  $\oplus((u,v,x),(u',v',y'),(u'',v'',z''))$ holds if there exist $y,z\in H$ such that \\
$(u,v,y)\sim(u',v',y')$ and $(u,v,z)\sim(u'',v'',z'')$, and $F_{(u,v)}\models x+y = z$.

\item  $\odot((u,v,x),(u',v',y'),(u'',v'',z''))$ holds if there exist $y,z\in H$ such that \\
$(u,v,y)\sim(u',v',y')$ and $(u,v,z)\sim(u'',v'',z'')$, and $F_{(u,v)}\models x\cdot y = z$.
\end{enumerate}

Lemma \ref{Existential} yielded a finitary existential formula defining the relation $(u,v,x)$ $\sim (u',v',x')$.
Moreover, the field addition and multiplication were defined in $F_{(u,v)}$ by finitary existential
formulas using $u$ and $v$, which were parameters there but here are elements of the triples in $D$.
Finally, we must consider the negations of the relations.
First, $(u,v,x)\not\sim (u',v',x')$ if and only if some $y'$ commuting with $u'$ and $v'$
satisfies $(u,v,x)\sim (u',v',y')$ and $y'\neq x'$ -- that is, just if $f_{(u,v),(u',v')}$ maps $x$
to some element different from $x'$.  Likewise, since $+$ is a binary operation in $F_{(u,v)}$,
the negation of $\oplus((u,v,x),(u',v',y'),(u'',v'',z''))$ is defined by
saying that some $w''\neq z''$ is the sum:
$$\exists w''([w'',u'']=1=[w'',v'']~\&~w''\neq z''~\&~\oplus((u,v,x),(u',v',y'),(u'',v'',w''))),$$
which is also existential, and similarly for the negation of $\odot$.\ 
Therefore, all of these sets have finitary existential definitions in the language of groups,
with no parameters, as do the negations of $\sim$, $\oplus$, and $\odot$.
(In fact, the complement of $D$ is $\Sigma_1$ as well.)

The functoriality of the system of isomorphisms $f_{(u,v),(u',v')}$ (across all pairs of
pairs of noncommuting elements) ensures that $\sim$ will be an equivalence relation.
Lemma \ref{Functorial} showed that $f_{(u,v),(u,v)}$ is always the identity,
giving reflexivity.  Transitivity follows from the functorial property
in that same lemma:
$$f_{(u,v),(u'',v'')} = f_{(u',v'),(u'',v'')}\circ f_{(u,v),(u',v')},$$
and with $(u'',v'')=(u,v)$, this property also yields the symmetry of $\sim$.

The definitions of $\oplus$ and $\odot$ essentially say to convert all three
triples into $\sim$-equivalent triples with the same initial coordinates $u$ and $v$,
and then to check whether the final coordinates satisfy Maltsev's definitions
of $+$ and $\cdot$ in the field $F_{(u,v)}$.  Understood this way, they clearly
respect the equivalence $\sim$.  Finally, by fixing any single noncommuting pair $(u,v)$,
we see that the set $\{(u,v,x)~:~x\in Z(H)\}$ contains one element from each
$\sim$-class and, under $\oplus$ and $\odot$, is isomorphic to the field $F_{(u,v)}$
defined by Maltsev, which in turn is isomorphic to the original field $F$.
\end{proof}

It should be noted that, although this interpretation of $F$ in $H(F)$ was
developed using computable functors on countable fields $F$, it is valid
even when $F$ is uncountable (or finite).  A full proof requires checking that
the system of isomorphisms $f_{(u,v),(u',v')}$ remains functorial and existentially definable
even in the uncountable case, but this is straightforward.

In Theorem \ref{thm:parameterfree}, to eliminate parameters from
Maltsev's definition of $F$ in $H(F)$, we gave an interpretation of $F$
in $H(F)$, rather than another definition.  (Recall that a definition is an interpretation
in which the equivalence relation on the domain is simply equality.)
We now demonstrate the impossibility of strengthening the theorem
to give a parameter-free definition of $F$ in $H(F)$.

\begin{prop}
\label{prop:nopfree}
There is no parameter-free definition of any field $F$ in its Heisenberg group $H(F)$
by finitary formulas.
\end{prop}
\begin{proof}
Suppose that there were such a definition, and let $D\subseteq (H(F))^n$ be its domain.
By Corollary \ref{cor:orbits}, the only $(x_1,\ldots,x_n)\in (H(F))^n$ that is fixed
by all automorphisms of $H(F)$ is the tuple where every $x_i$ is the identity element of $H(F)$.
So, for every $\vec x\in D$ except this identity tuple, there would be an
$\alpha_{\vec x}\in\text{Aut}(H(F))$ that does not fix $\vec x$.  With equality
of $n$-tuples as the equivalence relation on $D$, $\alpha_{\vec x}$ yields
an automorphism of the field $F$
(viewed as $D$ under the definable addition and multiplication)
that does not fix $\vec x$. However, both identity elements $0$ and $1$
in $F$ must be fixed by every automorphism of $F$.
\end{proof}















\section{Question of bi-interpretability}
\label{sec:biinterpretability}

If $\mathcal{B}$ is interpreted in $\mathcal{A}$, we write $\mathcal{B}^\mathcal{A}$ for the copy of $\mathcal{B}$ given by the interpretation of $\mathcal{B}$ in $\mathcal{A}$.  The structures $\mathcal{A}$ and $\mathcal{B}$ are \emph{effectively bi-interpretable} if there are uniformly relatively computable isomorphisms $f$ from $\mathcal{A}$ onto $\mathcal{A}^{\mathcal{B}^\mathcal{A}}$ and $g$ from $\mathcal{B}$ onto $\mathcal{B}^{\mathcal{A}^\mathcal{B}}$.  In general, the isomorphism $f$ would map each element of $\mathcal{A}$ to an equivalence class of equivalence classes of tuples in $\mathcal{A}$.  We would represent $f$ by a relation $R_f$ that holds for $a,\bar{a}_1,\ldots,\bar{a}_r$ if $f$ maps $a$ to the equivalence class of the tuple of equivalence classes of the $\bar{a}_i$'s.  Similarly, the isomorphism $g$ would be represented by a relation $R_g$ that holds for $b,\bar{b}_1,\ldots,\bar{b}_r$ if $g$ maps $b$ to the equivalence class of the tuple of equivalence classes of the $\bar{b}_i$'s.  Saying that $f$ and $g$ are uniformly relatively computable is equivalent to saying that the relations $R_f$, $R_g$, have generalized computable $\Sigma_1$ definitions without parameters.     

For a field $F$ and its Heisenberg group $H(F)$, when we define $H(F)$ in $F$, the elements of $H(F)$ are represented by triples from $F$, and we have finitary formulas, quantifier-free or existential, that define the group operation (as a relation).  When we interpret $F$ in $H(F)$, the elements of $F$ are represented by triples from $H(F)$, and we have finitary existential formulas that define the field operations and their negations (as ternary relations). Thus, in $F^{H(F)^F}$ (the copy of $F$ interpreted in the copy of $H(F)$ that is defined in $F$), the elements are equivalence classes of triples of triples.  In $H(F)^{F^{H(F)}}$ (the copy of $H(F)$ defined in the copy of $F$ that is interpreted in $H(F)$), the elements are triples of equivalence classes of triples.  So, an isomorphism $f$ from $F$ to $F^{H(F)^F}$ is represented by a $10$-ary relation $R_f$ on $F$, and an isomorphism $g$ from $H(F)$ to $H(F)^{F^{H(F)}}$---it is represented by a $10$-ary relation $R_g$ on $H(F)$.              

For a Turing computable embedding $\Theta$ of $K$ in $K'$ we have \emph{uniform effective bi-interpretability} if there are (generalized) computable $\Sigma_1$ formulas with no parameters that, for all $\mathcal{A}\in K$ and $\mathcal{B} = \Theta(\mathcal{A})$, define isomorphisms from $\mathcal{A}$ to $\mathcal{A}^{\mathcal{B}^\mathcal{A}}$ and from $\mathcal{B}$ to $\mathcal{B}^{\mathcal{A}^\mathcal{B}}$.  After a talk by the fifth author, Montalb\'{a}n asked the following very natural question.

\begin{ques} 
\label{ques:Montalban}
Do we have uniform effective bi-interpretability of $F$ and $H(F)$?

\end{ques}

The answer to this question is negative.  In particular, $\mathbb{Q}$ and $H(\mathbb{Q})$ are not effectively bi-interpretable.  One way to see this is to note that $\mathbb{Q}$ is rigid, while $H(\mathbb{Q})$ is not---in particular, for any non-commuting pair, $u,v\in H(\mathbb Q)$, there is a group automorphism that takes $(u,v)$ to $(v,u)$.
The negative answer to Question \ref{ques:Montalban} then follows from \cite[Lemma VI.26(4)]{MontalbanCST},
which states that if $\mathcal{A}$ and $\mathcal{B}$ are effectively bi-interpretable,
then their automorphism groups are isomorphic.

Morozov's result shows which half of effective bi-interpretability causes the difficulties.

\begin{prop} [Morozov]

There is a finitary existential formula that, for all $F$, defines in $F$ a specific isomorphism $k$ from $F$ to $F^{{H(F)}^F}$.  

\end{prop}

\begin{proof}

In $F$, we have the copy of $H(F)$, consisting of triples $(a,b,c)$ (representing $h(a,b,c)$), for 
$a,b,c\in F$.  The group operation, derived from matrix multiplication, is 
$(a,b,c)*(a',b',c') = (a+a',b+b',c+c'+ab')$.  The definitions of the universe and the operation are quantifier-free, with no parameters.  We have seen how to interpret $F$ in $H(F)$ using finitary existential formulas with no parameters.  There is a natural isomorphism $k$ from $F$ onto $F^{{H(F)}^F}$ obtained as follows.  In $H(F)$, let $u = h(1,0,0)$ and $v = h(0,1,0)$.  Then $\Delta_{(u,v)} = 1$.  We have an isomorphism mapping $F$ to $F_{(u,v)}$ that takes $\alpha$ to $h(0,0,\alpha)$.\  We let $k(\alpha)$ be the $\sim$-class of $(u,v,h(0,0,\alpha))$.\  The isomorphism $k$ is defined in $F$ by an existential formula.  The complement of $k$ is defined by saying that $k(\alpha)$ has some other value.
\end{proof}

The other half of what we would need for uniform effective bi-interpretability is sometimes impossible,
as remarked above in the case $F=\mathbb Q$.  We do not know of any examples where $F$ and $H(F)$
are effectively bi-interpretable:  the obstacle for $\mathbb Q$ might hold in all cases.

\begin{prob}

For which fields $F$, if any, are the automorphism groups of $F$ and $H(F)$ isomorphic?

\end{prob} 

Even if there are fields $F$ such that Aut$(F)\cong\text{Aut}(H(F))$,
we suspect that $F$ and $H(F)$ are not effectively bi-interpretable, simply because
it is difficult to see how one might give a computable $\Sigma_1$ formula in the language of groups
that defines a specific isomorphism from $H(F)$ to $H(F)^{F^{H(F)}}$.

\section{Generalizing the method}
\label{sec:generalize}   

Our first general definition and proposition follow closely the example
of a field and its Heisenberg group.

\begin{defn}
\label{defn:defn}

Let $\mathcal{A}$ be a structure for a computable relational language. Assume that its basic relations are $R_i$, where $R_i$ is $k_i$-ary.  We say that $\mathcal{A}$ is \emph{effectively defined in $\mathcal{B}$ with parameters} 
$\bar{b}$ if there exist $D(\bar{b})\subseteq\mathcal{B}^{<\omega}$, and $\pm R_i(\bar{b})\subseteq D(\bar{b})^{k_i}$, defined by a uniformly computable sequence of generalized computable $\Sigma_1$ formulas with parameters $\bar{b}$.  

\end{defn}   

\begin{prop}
\label{general}

Suppose $\mathcal{A}$ is effectively defined in $\mathcal{B}$ with parameters $\bar b$.
For $\bar{c}$ in the orbit of $\bar{b}$, let $\mathcal{A}_{\bar{c}}$ be the copy of $\mathcal{A}$ defined by the same formulas, but with parameters $\bar{c}$ replacing $\bar{b}$.  Then the following conditions together suffice to give an effective interpretation of $\mathcal{A}$ in $\mathcal{B}$ without parameters:  

\begin{enumerate}

\item  The orbit of $\bar{b}$ is defined by a computable $\Sigma_1$ formula $\varphi(\bar{u})$;

\item There is a generalized computable $\Sigma_1$ formula $\psi(\bar{u},\bar{v},\bar x,\bar y)$
such that for all $\bar{c},\bar{d}$ in the orbit of $\bar{b}$, the formula $\psi(\bar{c},\bar{d},\bar x,\bar y)$
defines an isomorphism $f_{\bar{c},\bar{d}}$ from $\mathcal{A}_{\bar{c}}$ onto $\mathcal{A}_{\bar{d}}$; and

\item  The family of isomorphisms $f_{\bar{c},\bar{d}}$ preserves identity and composition.  

\end{enumerate}   

\end{prop}    

\begin{proof}

We write $D(\bar{b})$, $\pm R_i(\bar{b})$ for the set and relations that give a copy of $\mathcal{A}$ and for the defining formulas (with parameters $\bar{b}$).  We obtain a parameter-free interpretation of $\mathcal{A}$ in $\mathcal{B}$ as follows:

\begin{enumerate}

\item  Let $D$ consist of the tuples $(\bar{c},\bar x)$ such that $\bar{c}$ is in the orbit of $\bar{b}$
and $\bar x$ is in $D(\bar{c})$.  This is defined by a generalized computable $\Sigma_1$ formula.    

\item  Let $\sim$ be the set of pairs $((\bar{c},\bar x),(\bar{d},\bar y))$ in $D^2$
such that $f_{\bar{c},\bar{d}}(\bar x) = \bar y$.  This is defined by a generalized computable $\Sigma_1$ formula.
For pairs $(\bar{c},\bar x)$, $(\bar{d},\bar y)$ from $D$, it follows that $(\bar{c},\bar x)\not\sim(\bar{d},\bar y)$
if and only if
$$(\exists \bar{y}')( (\bar d, \bar y') \in D~\&~f_{\bar{c},\bar{d}}(\bar x) = \bar{y}'\ \&\ \bar{y}'\not= \bar y).$$
Hence the negation of $\sim$ is also defined by a generalized computable $\Sigma_1$ formula.       

\item  We let $R_i^*$ be the set of $k_i$-tuples 
$((\bar{b}_1,\bar x_1),\ldots,(\bar{b}_{k_i},\bar x_{k_i}))$ in $D^{k_i}$ such that for the tuple
$(\bar y_1,\ldots,\bar y_{k_i})$ with $f_{\bar{b}_j,\bar{b}_1}(\bar x_j) = \bar y_j$, we have
$(\bar y_1,\ldots,\bar y_{k_i})\in R_i(\bar{b}_1)$.  This is defined by a generalized computable $\Sigma_1$ formula.
The complementary relation $\neg{R_i^*}$ is the set of tuples $((\bar{b}_1,\bar x_1),\ldots,(\bar{b}_{k_i},\bar x_{k_i}))$
such that for $\bar y_1,\ldots,\bar y_{k_i}$ as above, $(\bar y_1,\ldots,\bar y_{k_i})\in\neg{R_i(\bar{b}_1)}$.
This is also defined by a generalized computable $\Sigma_1$ formula.            

\end{enumerate}
The verification is identical to that of Theorem \ref{thm:parameterfree}.
\end{proof}

\begin{cor}
\label{cor:general}
In the situation of Proposition \ref{general}, if $D(\bar b)$ is contained in $\B^n$
for some single $n\in\omega$, then the $\psi$ in item (2) and
the formulas in Definition \ref{defn:defn} will simply be computable $\Sigma_1$ formulas
(as opposed to generalized computable $\Sigma_1$ formulas) and the interpretation of $\A$ in $\B$
without parameters will also be by computable (as opposed to generalized) $\Sigma_1$ formulas.
\qed\end{cor}

The reader will have noticed that we only produced an \emph{interpretation} of $\A$ in $\B$,
even though we originally had a \emph{definition} (with parameters) of $\A$ in $\B$.
Proposition \ref{prop:nopfree} shows that in general this is the best that can be done.  On the other hand,
we may extend Proposition \ref{general} and remove parameters even in the case where
$\mathcal{A}$ is interpreted (as opposed to being defined) with parameters in $\mathcal{B}$.  

\begin{defn} [Effective Interpretation with Parameters]

We say that $\mathcal{A}$, with basic relations $R_i$, $k_i$-ary, is \emph{effectively interpreted with parameters}
$\bar{b}$ in $\mathcal{B}$ if there exist $D\subseteq\mathcal{B}^{<\omega}$, $\equiv\subseteq D^2$,
and $R_i^*\subseteq D^{k_i}$ such that 

\begin{enumerate}

\item  $(D,(R_i^*)_i)/_\equiv\cong\mathcal{A}$, 

\item  $D$, $\pm\equiv$, and $\pm R_i^*$ are defined by a computable sequence of generalized computable $\Sigma_1$ formulas, with a fixed finite tuple of parameters $\bar{b}$.       

\end{enumerate}

\end{defn}

Again, in the case where $D\subseteq\mathcal{B}^n$ for some fixed $n$, the formulas defining the effective interpretation are computable $\Sigma_1$ formulas of the usual kind, with parameters $\bar{b}$.    

\begin{prop}
\label{prop:removeparameters}

Suppose that $\mathcal{A}$ (with basic relations $R_i$, $k_i$-ary) has an effective interpretation in $\mathcal{B}$
with parameters $\bar{b}$.  For $\bar{c}$ in the orbit of $\bar{b}$, let $\mathcal{A}_{\bar{c}}$ be the copy
of $\mathcal{A}$ obtained by replacing the parameters $\bar{b}$ by $\bar{c}$ in the defining formulas,
with domain $D_{\bar c}/\!\equiv_{\bar c}$ containing $\equiv_{\bar c}$-classes $[\bar a]_{\equiv_{\bar c}}$.  Then the following conditions suffice for an effective interpretation
of $\mathcal{A}$ in $\mathcal{B}$ (without parameters):

\begin{enumerate}  

\item  The orbit of $\bar{b}$ is defined by a computable $\Sigma_1$ formula $\varphi(\bar{x})$; 

\item  There is a relation $F\subseteq\mathcal{B}^{<\omega}$, with a generalized computable
$\Sigma_1$-definition, such that for every $\bar{c}$ and $\bar{d}$ in the orbit of $\bar{b}$, the set of pairs
$(\bar x,\bar y)\in D_{\bar c}\times D_{\bar d}$ with $(\bar{c},\bar{d},\bar{x},\bar{y})\in F$
is invariant under $\equiv_{\bar c}$ on $\bar x$ and under $\equiv_{\bar d}$ on $\bar y$, and
defines an isomorphism $f_{\bar{c},\bar{d}}$ from $\mathcal{A}_{\bar{c}}$ onto $\mathcal{A}_{\bar{d}}$; and


\item The family of isomorphisms $f_{\bar{c},\bar{d}}$ preserves identity and composition. 

\end{enumerate}

\end{prop}

\begin{proof}\
Let the new domain $D$ consist of those tuples $(\bar{c},\bar x)$ with $\bar{c}$
in the orbit of $\bar{b}$ and $\bar x$ in $D_{\bar{c}}$.  This is defined by a
generalized computable $\Sigma_1$ formula.    

Let the equivalence relation $\sim$ on $D$ be the set of pairs $((\bar{c},\bar x),(\bar{d},\bar y))\in D^2$ such that
$f_{\bar{c},\bar{d}}([\bar x]_{\equiv_{\bar c}}) = [\bar y]_{\equiv_{\bar d}}$.
This is defined by a generalized computable $\Sigma_1$ formula.  For $(\bar{c},\bar x)$, $(\bar{d},\bar y)\in D$,
we have $(\bar{c},\bar x)\not\sim(\bar{d},\bar y)$ if and only if
$$(\exists \bar y'\in D_{\bar d})~(f_{\bar{c},\bar{d}}([\bar x]_{\equiv_{\bar c}}) = [\bar y']_{\equiv_{\bar d}}~\&~\bar y\not\equiv_{\bar d}\bar y').$$
Hence $\not\sim$ is also defined by a generalized computable $\Sigma_1$ formula.       

Let $R_i^*$ be the set of $k_i$-tuples 
$((\bar{b}_1,\bar x_1),\ldots,(\bar{b}_{k_i},\bar x_{k_i}))$ in $D^{k_i}$ such that for the tuple
$(\bar y_1,\ldots,\bar y_{k_i})$ with 
$f_{\bar{b}_j,\bar{b}_1}(\bar x_j) = \bar y_j$, we have $(\bar y_1,\ldots,\bar y_{k_i})\in R_i(\bar{b}_1)$.
This is defined by a generalized computable $\Sigma_1$-formula.  The complementary relation
$\neg{R_i^*}$ is the set of tuples $((\bar{b}_1,\bar x_1),\ldots,(\bar{b}_{k_i},\bar x_{k_i}))$ such that for
$\bar y_1,\ldots,\bar y_{k_i}$ as above, $(\bar y_1,\ldots,\bar y_{k_i})\in\neg{R_i(\bar{b}_1)}$.
This too is defined by a generalized computable $\Sigma_1$ formula. 
Finally, as in the proofs of Theorem \ref{thm:parameterfree} and Proposition \ref{general},
it is clear that this yields an interpretation of $\A$ in $\B$ without parameters.
\end{proof}

A relation $R\subseteq\mathcal{B}^{<\omega}$ may have a definition that is \emph{generalized computable $\Sigma_\alpha$} for a computable ordinal $\alpha$, or \emph{generalized $X$-computable $\Sigma_\alpha$} for an $X$-computable ordinal $\alpha$, or \emph{generalized $L_{\omega_1\omega}$}, or \emph{generalized $\Sigma_\alpha$} for a countable ordinal $\alpha$.  The definition has the form $\bigvee_n \varphi_n(\bar{x}_n)$, where the sequence of disjuncts (each in $L_{\omega_1\omega}$, but of different arities $n$) is computable, or 
$X$-computable, or just countable.  We note that each generalized $L_{\omega_1\omega}$ formula is generalized $X$-computable $\Sigma_\alpha$ for an appropriately chosen $X$ and $\alpha$, and each generalized $\Sigma_\alpha$-formula is generalized $X$-computable $\Sigma_\alpha$ for an appropriately chosen $X$.

As computable structure theorists, we have focused here on effective interpretations.
Nevertheless, we wish to point out that our results apply not only to effective interpretations,
but to all interpretations using generalized $L_{\omega_1\omega}$ formulas.
The following theorem generalizes Proposition \ref{prop:removeparameters}
and considers every variation we can imagine.

\begin{thm}
\label{thm:fullygeneral}
Let $\mathcal{A}$ be a relational structure with basic relations $R_i$ that are $k_i$-ary. Suppose there is an interpretation of $\A$ in $\B$ by generalized $L_{\omega_1\omega}$ formulas,
with parameters $\bar b$ from $\B$. For $\bar{c}$ in the orbit of $\bar{b}$,
let $\mathcal{A}_{\bar{c}}$ be the copy of $\mathcal{A}$ obtained by the interpretation
with parameters $\bar{c}$ replacing $\bar{b}$.  Assume that there is a 
generalized $L_{\omega_1\omega}$-definable relation $F$
defining, for each $\bar c$ and $\bar d$ in the orbit of $\bar b$, an
isomorphism $f_{\bar{c},\bar{d}}:\mathcal{A}_{\bar{c}}\to\mathcal{A}_{\bar{d}}$
as in Proposition \ref{prop:removeparameters}, and that this family
is closed under composition, with the identity map as $f_{\bar c,\bar c}$ for all $\bar c$.

Then there is an interpretation of $\A$ in $\B$ by $L_{\omega_1\omega}$ formulas
without parameters.  Moreover, the new interpretation satisfies all of the following.
\begin{itemize}
\item
For each countable ordinal $\alpha$, if the interpretation in $(\B,\bar b)$ defines
$D$, $\equiv$, and each $R_i$ using $\Sigma_\alpha$
formulas from $L_{\omega_1\omega}$, and $F$ and the orbit of $\bar b$
in $\B$ are both defined by $\Sigma_\alpha$ formulas, then the parameter-free interpretation
also uses $\Sigma_\alpha$ formulas to define these sets.

\item
For each countable ordinal $\alpha$, if the interpretation in $(\B,\bar b)$ defines
each of $D$, $\pm\!\equiv$, and $\pm R_i$ using $\Sigma_\alpha$
formulas, and $F$ and the orbit of $\bar b$
in $\B$ are both defined by $\Sigma_\alpha$ formulas, then the parameter-free interpretation
also uses $\Sigma_\alpha$ formulas to define its domain, its equivalence relation $\sim$,
the complement $\not\sim$, and its relations $\pm R_i$.  (Defining $\not\sim$ and $\neg R_i$
this way is required by the usual notion of effective $\Sigma_\alpha$ interpretation.)

\item
Let $X\subseteq\omega$.  If the interpretation in $(\B,\bar b)$ used $X$-computable
formulas, and $F$ and the orbit of $\bar b$
in $\B$ are both defined by $X$-computable formulas, then the parameter-free interpretation
also uses $X$-computable formulas.

Of course, for every countable set
of $L_{\omega_1\omega}$ formulas, there is an $X$ that computes them all.
If the signature of $\A$ is infinite, and the formulas for the interpretation of $\A$ in
$(\B,\bar b)$ are computable uniformly in $X$, then so are the formulas 
for the parameter-free interpretation of $\A$ in $\B$.

(With $X=\emptyset$, $X$-computable formulas are simply computable formulas.)

\item
If the interpretation in $(\B,\bar b)$ had domain contained in $\B^n$ for a single $n$,
so that the defining formulas for this interpretation and for $F$
in $\B$ are all in $L_{\omega_1\omega}$ (as opposed to generalized $L_{\omega_1\omega}$),
then the parameter-free interpretation also uses (non-generalized) $L_{\omega_1\omega}$ formulas,
and its domain is contained in $\B^{n+|\bar b|}$.

\item
If the interpretation in $(\B,\bar b)$ used finitary
formulas, and $F$ and the orbit of $\bar b$
in $\B$ are both defined by finitary formulas, then the parameter-free interpretation
also uses finitary formulas.
\end{itemize}
\end{thm}

\begin{proof}
We obtain the parameter-free
interpretation just as in the proof of Proposition \ref{prop:removeparameters}.
Notice that, by a result of Scott in \cite{Scott}, the orbit of $\bar b$ must be
definable by some $L_{\omega_1\omega}$ formula.
Checking the specific claims is simply a matter of writing out the new formulas
using the old ones.
\end{proof}

\end{document}